\documentclass[12pt]{amsart}
\usepackage{amssymb}
\usepackage{graphicx}

\newcommand{\eps}{\varepsilon}

\newcommand{\N}{{\mathbb N}}
\newcommand{\C}{{\mathbb C}}

\newcommand{\R}{{\mathbb R}}

\newcommand{\tef}{transcendental entire function}

\newcommand{\shf}{subharmonic function}

\newcommand{\sconn}{simply connected}
\newcommand{\mconn}{multiply connected}

\theoremstyle{plain}
\newtheorem{theorem}{Theorem}[section]

\newtheorem{lemma}[theorem]{Lemma}
\theoremstyle{definition}

\theoremstyle{remark}

\theoremstyle{problem}

\theoremstyle{example}

\begin{document}


\title[Baker's conjecture and Eremenko's conjecture]{Baker's conjecture and Eremenko's conjecture for functions with negative real zeros}

\author{P. J. Rippon}
\address{Department of Mathematics and Statistics \\
The Open University \\
   Walton Hall\\
   Milton Keynes MK7 6AA\\
   UK}
\email{p.j.rippon@open.ac.uk}

\author{G. M. Stallard}
\address{Department of Mathematics and Statistics \\
The Open University \\
   Walton Hall\\
   Milton Keynes MK7 6AA\\
   UK}
\email{g.m.stallard@open.ac.uk}

\thanks{Both authors are supported by the EPSRC grant EP/H006591/1.}




\begin{abstract}
We introduce a new technique that allows us to make progress on two long standing conjectures in transcendental dynamics: Baker's conjecture that a {\tef} of order less than $1/2$ has no unbounded Fatou components, and Eremenko's conjecture that all the components of the escaping set of an entire function are unbounded. We show that both conjectures hold for many {\tef}s whose zeros all lie on the negative real axis, in particular those of order less than $1/2$. Our proofs use a classical distortion theorem based on contraction of the hyperbolic metric, together with new results which show that the images of certain curves must wind many times round the origin.
\end{abstract}

\maketitle

\section{Introduction}
\setcounter{equation}{0} Let $f:\C\to \C$ be a {\tef} and denote by
$f^{n},\,n=0,1,2,\ldots\,$, the $n$th iterate of~$f$. The {\it
Fatou set} $F(f)$ is the set of points $z \in \C$
such that $(f^{n})_{n \in \N}$ forms a normal
family in some neighborhood of $z$.  The complement of $F(f)$ is
called the {\it Julia set} $J(f)$ of $f$. An introduction to the
properties of these sets can be found in~\cite{wB93}.

This paper concerns two conjectures in transcendental dynamics. Eremenko's conjecture, arising from his paper~\cite{E} in 1989, is that the escaping set
\[
I(f) = \{z : f^n(z) \to \infty \mbox{ as } n \to \infty\}
\]
has no bounded components. This conjecture has motivated much current work in transcendental dynamics and it has become apparent that the escaping set plays a key role in the subject. For partial results on Eremenko's conjecture, see~\cite{lR07},~\cite{RRRS} and~\cite{RS05}, for example.

Baker's conjecture, arising from his paper~\cite{iB81} in 1981, is that the Fatou set has no unbounded components whenever the order of the function is less than~$1/2$ or even whenever the function has order at most $1/2$ minimal type. It is known~\cite{Z} that such functions have no unbounded periodic or preperiodic Fatou components but it remains open as to whether such a function can have an unbounded wandering domain, that is, an unbounded component $U$ of the Fatou set such that $f^n(U) \cap f^m(U) = \emptyset$ for $n \neq m$.

Many authors have shown that Baker's conjecture is satisfied provided some regularity condition is imposed on the growth of the maximum modulus but, without any such condition, it is not even known whether the conjecture holds for all functions of order zero. The strongest results in this direction are given in~\cite{HM} and in~\cite{RS08}. A survey of earlier work on this conjecture appears in~\cite{H}.

In this paper we prove that Baker's conjecture holds for a family of symmetric entire functions~$f$ of order less than $1/2$, without any restriction on the regularity of the growth of $M(r)$. To do this we introduce a completely new technique based on the winding of certain image curves. In particular, our new results cover all the examples of functions of order less than $1/2$ that were constructed in~\cite{HM} and~\cite{RS08} as examples to which the techniques of those papers {\it cannot} be applied.

All existing techniques for attacking Baker's conjecture use a `repeated stretching' technique based on the relationship between the maximum modulus and the minimum modulus, which are defined as follows:
\[
 M(r) = M(r,f) = \max_{|z| = r} |f(z)|, \quad m(r) = m(r,f) = \min_{|z| = r} |f(z)|.
\]
For functions of order less than $1/2$ the $\cos \pi\rho$ theorem implies that $m(r)$ is relatively large compared to $M(r)$ for many values of $r$. If, in addition, $M(r)$ is very small or has a certain regularity (see~\cite[Theorems~4 and~5]{RS08}, for example), then it can be shown that the forward images under $f$ of a long curve in an unbounded Fatou component experience repeated radial stretching which contradicts the contraction property of the hyperbolic metric.

As pointed out in~\cite{RS08} and~\cite{RS09a}, this stretching property based on the minimum modulus also implies that the escaping set has a certain `spider's web' structure, described below, which is sufficient to show that Baker's conjecture and also Eremenko's conjecture hold. More precisely, this property implies that the following subset of the escaping set has this spider's web structure:
\[
A_R(f) = \{z:|f^n(z)| \geq M^n(R), \text{ for } n \in \N\},
\]
where $M^n(r)$ denotes the $n$th iterate of $M$ with respect to $r$ and $R>0$ is chosen so that $M(r) > r$ for $r\geq R$. The set $A_R(f)$ is a subset of the {\it fast escaping set} $A(f) = \bigcup_{n \in \N} f^{-n}(A_R(f))$ which has many nice properties -- see~\cite{RS09a}.

We say that a set $E$ is a {\it spider's web} if $E$ is connected and there exists a sequence of simply connected domains $G_n$ such that
\[
\partial G_n \subset E, \; G_n \subset G_{n+1} \mbox{ for } n \in \N \;\mbox{ and }\; \bigcup_{n \in \N}G_n = \C.
\]
As shown in~\cite{RS09a}, if $A_R(f)$ is a spider's web, then $I(f)$ is a spider's web, so Eremenko's conjecture holds, and $f$ has no unbounded Fatou components, so Baker's conjecture holds.

Many functions $f$ of order less than $1/2$ have the property that $A_R(f)$ is a spider's web \cite[Theorem~1.9]{RS09a}, so it is tempting to conjecture that this is the case for {\it all} functions of order less then $1/2$. However, in~\cite{RS11} we show that there are functions of order less than $1/2$, and even of order~0, for which $A_R(f)$ is {\it not} a spider's web. The examples in~\cite{RS11} are also covered by the results in this paper.

The idea behind our proof is that for certain symmetric entire functions $f$ the forward images under $f$ of a long curve in an unbounded Fatou component must experience either repeated radial stretching or large winding round the origin, either of which leads to a contradiction. Our proof also shows that $I(f)$ is a spider's web, so Eremenko's conjecture holds for such functions.

Although the techniques we use here to obtain the winding of image curves rely on particular properties of the symmetric functions considered, it is plausible that ideas of this type could be applied to more general types of entire functions.

Recall that the order $\rho=\rho(f)$ of a {\tef} $f$ is
\[
\rho=\limsup_{r\to\infty}\frac{\log\log M(r,f)}{\log r}\,.
\]
Any function of order less than $1/2$, and indeed of order less than~$1$, has the form
\begin{equation}\label{prod}
f(z)=cz^{p_0}\prod_{n=1}^{\infty}\left(1+\frac{z}{a_n}\right)^{p_n},
\end{equation}
where $p_n\in \{0,1,\ldots\}$, for $n\ge 0$, $c \in \C$ and the sequence $(a_n)$ consists of complex numbers with $|a_n|$ strictly increasing. We show that both Baker's conjecture and Eremenko's conjecture hold for all such functions of order less than $1/2$ in the case that $c\in \R\setminus\{0\}$ and $a_n > 0$ for $n \in \N$. In fact, we prove the following more general result.

\begin{theorem}\label{main}
Let $f$ be a {\tef} of the form
\begin{equation}\label{fproduct}
f(z)=cz^{p_0}\prod_{n=1}^{\infty}\left(1+\frac{z}{a_n}\right)^{p_n},
\end{equation}
where $c\in \R\setminus\{0\}$, $p_n\in \{0,1,\ldots\}$, for $n\ge 0$, and the sequence $(a_n)$ is positive and strictly increasing. In addition, suppose that there exist $m>1$ and $R_0>0$ such that, for all $r \geq R_0$,
\begin{equation}\label{min}
 \mbox{ there exists } \rho \in (r,r^m) \mbox{ with } m(\rho) \geq M(r).
\end{equation}
Then
\begin{itemize}
\item[(a)]
there are no unbounded components of $F(f)$;
\item[(b)]
the set $I(f)$ is a spider's web, and hence is connected;
\item[(c)]
if, in addition,~$f$ has no {\mconn} Fatou components, then
$J(f)$ and $I(f)\cap J(f)$ are both spiders' webs.
\end{itemize}
\end{theorem}

{\it Remarks}\;\;1. There are many classes of functions for which the hypothesis \eqref{min} is satisfied - see, for example,~\cite[Section 8]{RS09a}. In particular, it is satisfied whenever~$f$ has order less than $1/2$; this was proved by Baker \cite[Satz 1]{iB58} and also follows from the version of the $\cos \pi \rho$ theorem proved by Barry~\cite{Ba}.

2. The proof of Theorem~\ref{main} uses two properties of the symmetric functions~$f$ under consideration, namely, that $f(\overline z)=\overline{f(z)}$, for $z\in\C$, and, for each $r>0$, $|f(re^{i\theta})|$ is strictly decreasing for $0\le \theta\le \pi$. Entire functions $f$ of the form $f(z)=g(z^q)$, where $g$ is a function of the form considered in Theorem~\ref{main} and $q\in\N$, have similar properties in relation to sectors of opening $2\pi/q$ and the proof of Theorem~\ref{main} can be modified to give the same conclusions for these functions. We omit the proof of this more general case.

Similar techniques can also be used to show that the result of Theorem~\ref{main} holds for functions of the same form as those in Theorem~\ref{main} but with, in addition, a finite number of zeros on the positive real axis.

3. In~\cite{RS11} it is shown that there are functions which satisfy the hypotheses of Theorem~\ref{main} and for which $A_R(f)$ and $A(f)$ both fail to have the structure of a spider's web. By part (b), these are examples of entire functions for which $I(f)$ is a spider's web but $A(f)$ is not a spider's web -- this answers a question in~\cite{RS09a}.

4. To prove part~(a), it is sufficient by a result of Zheng \cite{Z} to show that such functions have no unbounded wandering domains. In future work we plan to build on the techniques in this paper to show that there are no unbounded wandering domains for functions of order less than~$1$ of the type considered in Theorem~\ref{main}, even when condition~\eqref{min} is not satisfied. This raises the question as to whether Baker's conjecture should be strengthened to say, for example, that a function of order less than~$1$ has no unbounded wandering domains. We note that there are functions of order~$1/2$ with unbounded periodic Fatou components; for example, $f(z)=(\sinh\sqrt{z})/\sqrt{z}$ has an attracting basin containing the negative real axis, and other such examples are given in~\cite{iB81}.\\

The structure of this paper is as follows. In Section~2, we state three main results, Theorems~\ref{main1},~\ref{main2} and~\ref{main3}. The first two concern the winding of image curves, and the third is a more precise version of Theorem~\ref{main}. Theorem~\ref{main2} is deduced easily from Theorem~\ref{main1} in Section~2. Section~3 is devoted to proving Theorem~\ref{main1} and Section~4 contains some general results needed for the proof of Theorem~\ref{main3}, which is given in Section~5.

\section{Methodology}
\setcounter{equation}{0}
In this section we state three theorems that together give greater insight into the method that we use to prove Theorem~\ref{main}. The first two theorems show that for entire functions of the type covered by Theorem~\ref{main} the images of certain continua~$\gamma$ must wind many times round the origin, while the third states a number of detailed properties of such functions, including those given in Theorem~\ref{main}. We note that for our present purposes it would be sufficient to prove the first two theorems in the case when~$\gamma$ is a curve rather than a continuum, but we state these theorems for a continuum because of possible future applications.

We begin by introducing some notation. If $\gamma$ is a plane curve with an associated parametrisation and $\gamma$ meets no zeros of~$f$, then we denote the net change in the argument of $f(z)$ as $z$ traverses $\gamma$ by $\Delta\!\arg f(\gamma)$.

If $\gamma$ is a plane continuum having the property that all the zeros of~$f$ lie in the unbounded complementary component of $\gamma$, and $z_0,z'_0$ is any pair of distinct points in~$\gamma$, then we denote the net change in the argument of $f(z)$ as $z$ traverses~$\gamma$ from $z_0$ to $z'_0$ by $\Delta\!\arg (f(\gamma);z_0,z'_0)$. This quantity is defined by choosing
\begin{itemize}
\item
any {\sconn} domain, $G$ say, that contains $\gamma$ but no zeros of $f$,
\item a branch, $g$ say, of $\log f$ in~$G$,
\end{itemize}
and putting
\begin{equation}\label{Deltadef}
\Delta\!\arg (f(\gamma);z_0,z'_0)=\Im (g(z'_0))-\Im (g(z_0)).
\end{equation}

Next, for $r>0$, we write $C(r) = \{z: |z| = r\}$ and, for $0<r_1<r_2$, we write
\[
A(r_1,r_2)=\{z:r_1<|z|<r_2\}\;\;\text{and}\;\;\overline A(r_1,r_2)=\{z:r_1\le |z|\le r_2\}.
\]

We also recall the following version of Hadamard convexity: for a {\tef} $f$, there exists a constant $R_1=R_1(f)>0$ such that
\begin{equation}\label{Had}
M(r^c)\ge M(r)^c, \;\;\text{for } r\ge R_1,\; c>1;
\end{equation}
see~\cite[Lemma~2.1]{RS08}. Throughout the paper $R_1$ denotes the constant in \eqref{Had}.
\begin{theorem}\label{main1}
Let $f$ be a {\tef} of the form
\begin{equation}\label{prod}
f(z)=cz^{p_0}\prod_{n=1}^{\infty}\left(1+\frac{z}{a_n}\right)^{p_n},
\end{equation}
where $c\in\R\setminus\{0\}$, $p_n\in \{0,1,\ldots\}$, for $n\ge 0$, and the sequence $(a_n)$ is positive and strictly increasing, and let $t$ and $a$ be positive numbers such that
\begin{equation}\label{tanda}
t\ge R_1,\;\;t^{a/4}\ge \frac{48\sqrt 2}{\pi}\; \mbox{ and }\; at^{a/4}\ge \frac{384\sqrt 2}{\pi}\,.
\end{equation}

If $\gamma$ is any continuum in $\{z: \Im z \geq 0\}$ that meets both $C(t)$ and $C(t^{1+a})$ with
\begin{equation}\label{fsmall}
 1/M(t) < |f(z)| < M(t), \mbox{ for } z \in \gamma,
\end{equation}
then there exists a continuum $\Gamma\subset \gamma\cap\overline A(t,t^{1+a})$ and $z_0, z'_0\in\Gamma$ such that
\[
\Delta\!\arg(f(\Gamma);z_0,z'_0) \geq  \frac{\pi^2 t^{a/4} \log M(t)}{48\sqrt 2 \log t}\,.
\]
\end{theorem}

{\it Remark}\quad We note that a special case of Theorem~\ref{main1} occurs when \eqref{fsmall} is replaced by the condition $|f|=\lambda$, where $1/M(t)<\lambda<M(t)$, in which case $\gamma$ is a level curve of $f$.

The proof of Theorem~\ref{main1} is given in Section~3. Our second theorem has an additional assumption on the function $f$; since this theorem follows easily from Theorem~\ref{main1} we give the proof here.
\begin{theorem}\label{main2}
Let $f$ be a {\tef} of the form
\begin{equation}\label{prod2}
f(z)=c z^{p_0}\prod_{n=1}^{\infty}\left(1+\frac{z}{a_n}\right)^{p_n},
\end{equation}
where $c\in\R\setminus\{0\}$, $p_n\in \{0,1,\ldots\}$, for $n\ge 0$, and the sequence $(a_n)$ is positive and strictly increasing. In addition, suppose that there exist $m>1$ and $R_0>0$ such that, for all $r \geq R_0$,
\begin{equation}\label{min2}
 \mbox{ there exists } \rho \in (r,r^m) \mbox{ with } m(\rho) \geq M(r),
\end{equation}
and let $s,l,b>0$ satisfy
\begin{equation}\label{conditions}
s\ge \max\{1,R_0,R_1\},\;l\geq \max\{m,1+b\},\; M(s) \geq s^2, \; s^{b/4}\ge \frac{48\sqrt 2}{\pi}\;\mbox{and}\;\frac{bs^{b/4}}{l-b}\ge \frac{384\sqrt 2}{\pi}\,.
\end{equation}

If $\gamma$ is any continuum in $\{z: \Im z \geq 0\}$ that meets both $C(s)$ and $C(s^l)$, then at least one of the following must hold:
\begin{itemize}
\item[(1)] $f(\gamma)$ meets both $C(M(s))$ and $C(M(s)^{l - b})$;

\item[(2)] $f(\gamma)$ meets both $C(M(s)^{1/l})$ and $C(M(s))$;

\item[(3)] there exists a continuum $\Gamma \subset \gamma \cap \overline A(s^{l-b},s^l)$ and $z_0, z'_0\in\Gamma$ such that
\[
f(\Gamma) \subset \overline A(M(s)^{1/l}, M(s)^{l-b})
\]
and
\[
\Delta\!\arg(f(\Gamma);z_0,z'_0) \geq  \frac{\pi^2 s^{b/4} \log M(s^{l -
b})}{48\sqrt 2 \log s^{l - b}} \geq  \frac{\pi^2 s^{b/4} \log M(s)}{48\sqrt 2  \log s}\ge 2\pi.
\]
\end{itemize}
\end{theorem}
\begin{proof}
Suppose that $f,s,l,b$ and $\gamma$ satisfy the hypotheses of the theorem. Since~$\gamma$ meets $C(s)$, there exists $z_1 \in \gamma$ such that $|f(z_1)| \leq M(s)$. So, if there exists $z_2 \in \gamma$ such that $|f(z_2)| \geq M(s)^{l-b}$, then case~(1) must hold.

Also, since $l \geq m$ and $s \geq R_0$, it follows from~\eqref{min2} that there exists $z_3 \in \gamma$ such that $|f(z_3)| \geq M(s)$. So, if there exists $z_4 \in \gamma$ such that $|f(z_4)| \leq M(s)^{1/l}$, then case~(2) must hold.

If neither case~(1) nor case~(2) holds, then it follows that
\begin{equation}\label{casec}
1 \leq M(s)^{1/l} < |f(z)| < M(s)^{l-b}\le M(s^{l-b}),\; \mbox{ for } z \in \gamma,
\end{equation}
by~\eqref{Had}, since $s\ge R_1$ and  $l-b\ge 1$.

We can now apply Theorem~\ref{main1} to $\gamma$ with $t=s^{l-b}$ and $a=b/(l-b)$, since $s^l = (s^{l-b})^{1 + b/(l-b)}$. The hypotheses of Theorem~\ref{main1} are satisfied since, in addition to~\eqref{casec}, we have $t=s^{l-b}\ge s\ge R_1$,
\[
t^{a/4}=(s^{(l-b)/4})^{b/(l-b)} = s^{b/4} \ge \frac{48\sqrt 2}{\pi}
\]
and
\[
at^{a/4}=\frac{bs^{b/4}}{l-b}\ge \frac{384\sqrt 2}{\pi}\,,
\]
by~\eqref{conditions}. We deduce, by Theorem~\ref{main1} and~\eqref{casec}, that there exists a continuum $\Gamma \subset \gamma \cap \overline A(s^{l-b},s^l)$ and $z_0, z'_0\in\Gamma$ such that
\[
f(\Gamma) \subset\overline A(M(s)^{1/l}, M(s)^{l-b})
\]
and
\[
\Delta\!\arg(f(\Gamma);z_0,z'_0) \geq  \frac{\pi^2 s^{b/4} \log M(s^{l -
b})}{48\sqrt 2 \log s^{l - b}}\,,
\]
so case~(3) holds; the final two inequalities follow from~\eqref{Had} and~\eqref{conditions}.
\end{proof}
Our third theorem is a more detailed version of Theorem~\ref{main}, which gives more precise information about the size of any {\sconn} Fatou components of~$f$ and the nature of the $I(f)$ spider's web. It is proved using Theorem~\ref{main2}.
\begin{theorem}\label{main3}
Let $f$ be a {\tef} of the form
\begin{equation}\label{fproduct}
f(z)=cz^{p_0}\prod_{n=1}^{\infty}\left(1+\frac{z}{a_n}\right)^{p_n},
\end{equation}
where $c\in\R\setminus\{0\}$, $p_n\in \{0,1,\ldots\}$, for $n\ge 0$, and the sequence $(a_n)$ is positive and strictly increasing. In addition, suppose that there exist $m>1$ and $R_0>0$ such that, for all $r \geq R_0$,
\begin{equation}\label{min3}
 \mbox{ there exists } \rho \in (r,r^m) \mbox{ with } m(\rho) \geq M(r).
\end{equation}
Then there exist $L>0$ and $R>0$ such that $f$ has the following properties.
\begin{itemize}
\item[(a)]
There are no unbounded components of $F(f)$; more precisely, every {\sconn} component of $F(f)$ that meets $\{z:|z|>R^L\}$ lies in an annulus of the form $A(r,r^L)$, $r>0$.
\item[(b)]
The set $I(f)$ is a spider's web, and hence is connected; more precisely, with $\mu(r)=M(r)^{\eps}$, where $\eps=1/L$, the subset of $I(f)$ defined as
\[
Q_{\eps,R}(f)=\{z:|f^n(z)|\ge \mu^n(R),\; n=0,1,\ldots\}
\]
has the following properties:
\begin{itemize}
\item[(i)]
every component of $Q_{\eps,R}(f)^c$ that meets $\{z:|z|>R^L\}$ lies in an annulus of the form $A(r,r^L)$, $r>0$;
\item[(ii)]
$Q_{\eps,R}(f)$ contains a spider's web.
\end{itemize}
\item[(c)]
If, in addition,~$f$ has no {\mconn} Fatou components, then
$J(f)$ and $I(f)\cap J(f)$ are both spiders' webs; more precisely, $Q_{\eps^2,R}(f)\cap J(f)$ contains a spider's web.
\end{itemize}
\end{theorem}
{\it Remark}\;\;The set $Q_{\eps,R}(f)$ is a subset of the {\it quite fast escaping set} $Q(f)$, introduced in~\cite{RS11b}, where it is shown that for many entire functions $Q(f) = A(f)$. These sets are not equal in general, however, as shown by examples in~\cite{RS11}.

We prove Theorem~\ref{main3} in Section~5 and give some results needed for the proof in Section~4.

\section{Proof of Theorem~\ref{main1}}
\setcounter{equation}{0}
First we state some preliminary results. For a function $u$ subharmonic in the closed disc $\{z:|z|\le r_0\}$, that is, subharmonic in some larger open disc, we define
\begin{equation}\label{Bdef}
B(r)=B(r,u)=\max_{|z|=r}u(z),\quad 0<r\le r_0,
\end{equation}
and
\begin{equation}\label{Tdef}
T(r)=T(r,u)=\frac{1}{2\pi}\int_0^{2\pi}u^+(re^{i\theta})\,d\theta,\quad 0< r\le r_0.
\end{equation}
Basic facts about $B(r)$ and $T(r)$, such as that they are increasing and convex with respect to $\log r$, can be found in \cite[Section 2.7]{HK}, for example. We also need the following consequence of the Poisson formula for a positive harmonic function in a disc; see \cite[Theorem~3.19]{HK}.
\begin{lemma}\label{Poisson}
Suppose that $u$ is subharmonic in $\{z:|z|\le r_0\}$. Then
\[
T(r,u)\le B(r,u)\le \left(\frac{r_0+r}{r_0-r}\right)\,T(r_0,u),\;\;\text{for } 0<r<r_0.
\]
\end{lemma}
Next we recall the Milloux-Schmidt inequality \cite[page~289]{wH89}.
\begin{lemma}\label{MS}
Suppose that $u$ is subharmonic in $\{z:|z|\le r_0\}$ and that
\begin{equation}\label{min0}
\min_{|z|=r}u(z)=0, \;\;\text{for } 0<r<r_0.
\end{equation}
Then
\[
B(r,u)\le \frac{4}{\pi}B(r_0,u)\tan^{-1}\left(\frac {r}{r_0}\right)^{1/2},\;\;\text{for } 0<r<r_0.
\]
\end{lemma}

Finally, we need a simple result in plane topology and include a proof for completeness.
\begin{lemma}\label{continuum}
Let $\gamma$ be a closed connected set in $\C$ that meets the circles $C(r_1)$ and $C(r_2)$, where $0<r_1<r_2$. Then $\gamma$ contains a continuum $\Gamma$ that meets $C(r_1)$ and $C(r_2)$ with $\Gamma\subset \overline A(r_1,r_2)$.
\end{lemma}
\begin{proof}
First let $\gamma_1$ be a component of $\gamma\cap \{z:|z|<r_2\}$ that meets $C(r_1)$. Then $\overline{\gamma_1}$ is a connected subset of $\gamma$ and meets $C(r_2)$ by \cite[page~84]{New}. Next let $\gamma_2$ be a component of $\overline{\gamma_1}\cap \{z:|z|>r_1\}$ that meets $C(r_2)$. Then $\Gamma=\overline{\gamma_2}$ is a connected subset of $\gamma$ that meets both $C(r_1)$ and $C(r_2)$ with $\Gamma\subset \overline A(r_1,r_2)$.
\end{proof}

Now we prove Theorem~\ref{main1}. Let $\gamma$ be a continuum lying in $\{z:\Im z\ge 0\}$ and meeting both $C(t)$ and $C(t^{1+a})$ such that
\begin{equation}\label{fsmall2}
 1/M(t) < |f(z)| < M(t), \;\mbox{ for } z \in \gamma.
\end{equation}
Then Lemma~\ref{continuum} implies that there is a subcontinuum $\Gamma$ of $\gamma$ lying in $\overline{A}(t, t^{1+a})$ and meeting both $C(t)$ and $C(t^{1+a})$. Clearly $\Gamma$ also lies in $\{z:\Im z\ge 0\}$.

Recall that $t$ and $a$ are assumed to satisfy the inequalities in~\eqref{tanda}. Put
\[
r_1=t\;\text{ and }\; r_2=t^{1+a},
\]
and, for $i=1,2$, let $C_i$ denote the largest circular arc with centre~0 of radius~$r_i$ such that $C_i$ meets the positive real axis but does not meet $\Gamma\cup \Gamma^*$, where $*$ denotes reflection in the real axis. The idea of the proof is to use the bounds in \eqref{fsmall2} together with the Milloux-Schmidt inequality to show that $M(r,f)$ must grow by a large amount between $r_1$ and $r_2$, and then use this growth to show that $f(C_2)$ must wind round~$0$ many more times than $f(C_1)$ does, from which it follows that $f(\Gamma)$ must wind many times round~$0$.

For $\lambda>0$, let
\[
u_{\lambda}(z)=\log^+(|f(z)|/\lambda),\quad z\in\C,
\]
and
\[
G_{\lambda}=\{z:|f(z)|>\lambda\}.
\]
Then $u_{\lambda}$ is a non-negative continuous {\shf} in $\C$, symmetric with respect to the real axis, which vanishes on $\partial G_{\lambda}$, and is positive harmonic in $G_{\lambda}$. As in \eqref{Bdef} and \eqref{Tdef}, let
\[
B_{\lambda}(r)=B(r,u_{\lambda}) = \max_{|z|=r} u_{\lambda}(z)=\log^+ (M(r)/\lambda),\quad r>0,
\]
and
\[
T_{\lambda}(r)=T(r,u_{\lambda}) =\frac{1}{2\pi}\int_0^{2\pi}u_{\lambda}(re^{i\theta})\,d\theta,\quad r>0.
\]

Let $M=M(t)$. For each fixed $r>0$, $|f(re^{i\theta})|$ is a strictly decreasing function of $\theta$ for $0\le \theta\le \pi$, by \eqref{prod}. Thus, for $\lambda\in (0,M)$, the set $G_{\lambda}$ meets each circle $C(r)$, $r\ge t$, in a circle or an open circular arc, $C_{\lambda}(r)$ say, that is symmetric with respect to the positive real axis. The next lemma relates the the number of times $f(C_{\lambda}(r))$ winds round~$0$ to the mean $T_{\lambda}$.
\begin{lemma}\label{winding1}
For $\lambda\in (0,M)$ and $r\ge t$, we have
\begin{equation}\label{winding1a}
\Delta\!\arg f(C_{\lambda}(r))=2\pi rT'_{\lambda}(r)>0.
\end{equation}
\end{lemma}
\begin{proof}
Since $C_{\lambda}(r)$ is either a circle or an arc in $G_{\lambda}$, at the endpoints of which $u_{\lambda}$ vanishes, we have
\[
rT'_{\lambda}(r)=\frac{r}{2\pi}\int_0^{2\pi}\frac{\partial u_{\lambda}}{\partial r}\,d\theta=\frac{1}{2\pi}\int_{C_{\lambda}(r)}\frac{\partial u_{\lambda}}{\partial r}\,ds=\frac{1}{2\pi}\int_{C_{\lambda}(r)}\frac{\partial v_{\lambda}}{\partial s}\,ds.
\]
Here $v_{\lambda}$ is a harmonic conjugate of $u_{\lambda}$ obtained by taking a branch of $\arg f$ defined on a neighbourhood of $C_{\lambda}(r)\setminus \{-r\}$. The result now follows since the last integral above gives the net change in the argument of $f(z)$ as $z$ traverses $C_{\lambda}(r)$ and $T(r)$ is an increasing convex function of $\log r$.
\end{proof}

For $i=1,2$, let $\lambda_i=\inf_{z\in C_i}|f(z)|$. Then, by \eqref{fsmall2},
\begin{equation}\label{Ci}
C_i=C_{\lambda_i}(r_i)\;\text{ and }\;\lambda_i\in (1/M,M),\;\; \text{for } i=1,2.
\end{equation}

The next lemma relates the difference between the winding round~0 of $f(C_2)$ and $f(C_1)$ to the growth of $M(r)$ between $r_1$ and $r_2$.
\begin{lemma}\label{wind-growth}
With the above notation we have
\[
\frac{1}{2\pi}\left(\Delta\!\arg f(C_2)-\Delta\!\arg f(C_1)\right)\ge \frac{\tfrac13B_{M}(r_2/2)-2B_{M}(s)-4\log M}{\tfrac12\log(r_2/r_1)}\,,
\]
where $s=\sqrt{r_1r_2}=t^{1+a/2}$.
\end{lemma}
\begin{proof}
For $\lambda\in (0,M)$, let
\begin{equation}\label{phidef}
\phi_{\lambda}(\rho)=T_{\lambda}(e^{\rho}), \quad \rho\ge \log t.
\end{equation}
Then each $\phi_{\lambda}$ is convex and increasing, and
\[
\phi'_{\lambda}(\rho)=rT_{\lambda}'(r),\quad\text{where }r=e^{\rho}\ge t.
\]
Thus \eqref{winding1a} can be written in the form
\begin{equation}\label{winding4}
\Delta\!\arg f(C_{\lambda}(r))=2\pi \phi'_{\lambda}(\rho),\;\;\text{where }\rho=\log r.
\end{equation}
Also put
\[
\rho_1=\log r_1,\;\;\rho_2=\log r_2\;\text{ and }\; \sigma=\log s.
\]

Since each $\phi_{\lambda}$ is convex and increasing and $\sigma=\tfrac12(\rho_1+\rho_2)$, we deduce from \eqref{Ci}, \eqref{phidef} and Lemma~\ref{Poisson} that
\begin{eqnarray*}
\phi'_{\lambda_1}(\rho_1)&\le&\frac{\phi_{\lambda_1}(\sigma)-\phi_{\lambda_1}(\rho_1)}{\tfrac12(\rho_2-\rho_1)}\\
&=& \frac{T_{\lambda_1}(s)-T_{\lambda_1}(r_1)}{\tfrac12\log(r_2/r_1)}\\
&\le& \frac{B_{\lambda_1}(s)}{\tfrac12\log(r_2/r_1)}\,.
\end{eqnarray*}
Similarly,
\begin{eqnarray*}
\phi'_{\lambda_2}(\rho_2)&\ge&\frac{\phi_{\lambda_2}(\rho_2)-\phi_{\lambda_2}(\sigma)}{\tfrac12(\rho_2-\rho_1)}\\
&=& \frac{T_{\lambda_2}(r_2)-T_{\lambda_2}(s)}{\tfrac12\log(r_2/r_1)}\\
&\ge& \frac{\tfrac13 B_{\lambda_2}(r_2/2)-B_{\lambda_2}(s)}{\tfrac12\log(r_2/r_1)}\,.
\end{eqnarray*}
Thus, by~\eqref{Ci} and~\eqref{winding4},
\begin{equation}\label{winding1b}
\frac{1}{2\pi}\left(\Delta\!\arg f(C_2)-\Delta\!\arg f(C_1)\right)\ge \frac{\tfrac13B_{\lambda_2}(r_2/2)-B_{\lambda_2}(s)-B_{\lambda_1}(s)}{\tfrac12\log(r_2/r_1)}\,.
\end{equation}

Now, for $r>t$ and $\lambda\in(0,M]$, we have
\[
B_{\lambda}(r)=\log f(r)-\log \lambda,
\]
so, by \eqref{Ci},
\[
B_{\lambda_2}(r_2/2)\ge B_M(r_2/2)\;\;\text{and}\;\;B_{\lambda_i}(s)\le B_M(s)+2\log M,\;\;\text{for }i=1,2.
\]
Hence, by \eqref{winding1b},
\[
\frac{1}{2\pi}\left(\Delta\!\arg f(C_2)-\Delta\!\arg f(C_1)\right)
\ge \frac{\tfrac13B_{M}(r_2/2)-2B_{M}(s)-4\log M}{\tfrac12\log(r_2/r_1)},
\]
as required.
\end{proof}

Next we estimate the growth of $B_M(r)$ from $s=t^{1+a/2}$ to $r_2/2=t^{1+a}/2$. Recall that
\[
G_M=\{z:|f(z)|>M\}\; \text{ and }\;u_M(z)=\log^+(|f(z)|/M),
\]
so $u_M$ vanishes on $\partial G_M$. By \eqref{fsmall2}, we have $\Gamma\cap \overline{G_M}=\emptyset$, so each circle $C(r)$, $t<r<t^{1+a}$, meets $\partial G_M$. Also, $|f(z)|\le M$, for $|z|\le t$, since $M=M(t)$. Thus
\[
\min_{|z|=r}u_M(z)=0, \quad\text{for } 0\le r < t^{1+a}.
\]
Therefore, by Lemma~\ref{MS} with $r=s=t^{1+a/2}$ and $r_0=r_2/2=t^{1+a}/2$, we have
\begin{equation}\label{Bgrowth1}
B_M(t^{1+a/2})\le \frac{4\sqrt 2}{\pi}\,t^{-a/4}B_M(t^{1+a}/2),
\end{equation}
since $r<r_0$, by the second inequality in \eqref{tanda}. Now
\[
B_M(t^{1+a/2})=\log (M(t^{1+a/2})/M),
\]
and, by~\eqref{Had}, since $t\ge R_1$,
\[
M(t^{1+a/2})\ge M(t)^{1+a/2}=M^{1+a/2}.
\]
Hence
\begin{equation}\label{Bgrowth2}
B_M(t^{1+a}/2)\ge \frac{\pi at^{a/4}\log M}{8\sqrt 2}\,.
\end{equation}

We can now deduce that $f(C_2)$ winds many more times round~$0$ than $f(C_1)$ does. By Lemma~\ref{wind-growth}, \eqref{Bgrowth1}, \eqref{Bgrowth2} and the last two inequalities in \eqref{tanda},
\begin{eqnarray*}
\frac{1}{2\pi}\left(\Delta\!\arg f(C_2)-\Delta\!\arg f(C_1)\right)
&\ge& \frac{\tfrac13B_{M}(r_2/2)-2B_{M}(s)-4\log M}{\tfrac12\log(r_2/r_1)}\\
&\ge& \frac{\tfrac13B_{M}(t^{1+a}/2)(1-(24\sqrt 2/\pi)t^{-a/4})-4\log M}{\tfrac12 a\log t}\\
&\ge& \frac{\pi at^{a/4}\log M/(48\sqrt 2)-4\log M}{\tfrac12 a\log t}\\
&\ge& \frac{\pi t^{a/4}\log M}{48\sqrt 2 \log t}.
\end{eqnarray*}

To deduce that $f(\Gamma)$ winds many times round~$0$, we introduce the function $F(z)=\log f(z)$ which is analytic on $\C\setminus (-\infty,0]$, the branch being chosen so that~$F$ is real on the positive real axis. Let $\overline{C_i}\cap \Gamma=\{z_i\}$, for $i=1,2$. Then~$F$ maps the parts of $C_1$ and $C_2$ in the upper half-plane to curves each with one endpoint on the positive real axis. If the other endpoints are denoted by $w_1=F(z_1)$ and $w_2=F(z_2)$, then, by the symmetry of $f$ in the real axis,
\[
\Im w_2-\Im w_1\ge  \frac{\pi^2 t^{a/4} \log M}{48\sqrt 2 \log t}\,.
\]
Since $F(\Gamma)$ joins $w_1$ to $w_2$ and $M=M(t)$, it follows that
\[
\Delta\!\arg(f(\Gamma);z_1,z_2) \geq  \frac{\pi^2 t^{a/4} \log M(t)}{48\sqrt 2 \log t} \,.
\]
This completes the proof of Theorem~\ref{main1}.

\section{Preliminary results for the proof of Theorem~\ref{main3}}
\setcounter{equation}{0}
To deduce Theorem~\ref{main3} from Theorem~\ref{main2} we need several preliminary results, all of which hold for a general {\tef}. First we need a standard distortion theorem for iterates in escaping Fatou components; see \cite[Lemma~7]{wB93}.
\begin{lemma}\label{dist}
Let $f$ be a {\tef}, let $U \subset I(f)$ be a {\sconn} Fatou component of $f$ and let $K$ be a compact subset of $U$.
There exist $C>1$ and $N \in \N$ such that
\[
|f^n(z_0)| \leq C|f^n(z_1)|, \; \mbox{ for } z_0, z_1 \in K, \; n \geq N.
\]
\end{lemma}

The next lemma is well known but we include a proof for completeness.
\begin{lemma}\label{connectivity}
Let $f$ be a {\tef}, let $U$ be a {\sconn} Fatou component of $f$, and let $U_n$, $n\in\N$, be the Fatou component of $f$ that contains $f^n(U)$. Then $U_n$ is {\sconn}.
\end{lemma}
\begin{proof}
 Suppose that $U_n$ is {\mconn}. Then it follows from a theorem of Baker~\cite[Theorem~3.1]{iB84} that $U$ is bounded. Thus $f^n:U\to U_n$ is a proper map, so we obtain a contradiction by the Riemann-Hurwitz formula \cite[Section~1.2]{nS93}.
\end{proof}

Next we collect together several key properties of the fast escaping set
\[A(f) = \bigcup_{n \in \N} f^{-n}(A_R(f)),\]
where
\[A_R(f)=\{z:|f^n(z)|\ge M^n(R),\;\text{for }n\in\N\},\]
and $R>0$ is such that $M(r) > r$ for $r\geq R$. The following properties can all be found in \cite{RS09a}.
\begin{lemma}\label{fast}
Let $f$ be a {\tef}. Then
\begin{itemize}
\item[(a)]
every component of $A_R(f)$ and $A(f)$ is unbounded;
\item[(b)]
$J(f)=\partial A(f)=\overline{A(f)\cap J(f)}$;
\item[(c)]
if, in addition, $f$ has no {\mconn} Fatou components, then every component of $A_R(f)\cap J(f)$ and $A(f)\cap J(f)$ is unbounded.
\end{itemize}
\end{lemma}
The next property of the fast escaping set is not stated explicitly in \cite{RS09a} but it follows easily from results given there.
\begin{lemma}\label{AR-component}
Let $f$ be a {\tef}. There exists $R_2=R_2(f)>0$ such that if $R\ge R_2$, then there is a component of $A_{R/2}(f)$ that meets $\{z:|z|<R\}$ and is unbounded.
\end{lemma}
\begin{proof}
By \cite[Theorem~2.4]{RS09a}, if $R>0$ is sufficiently large, then there exists $z'$ with $R/2<|z'|<R$ such that
\[
|f^n(z')|\ge M^n(R/2),\;\;\text{for }n\in \N.
\]
Thus $z'\in A_{R/2}(f)$, and so, by Lemma~\ref{fast} part~(a), $z'$ lies in an unbounded component of $A_{R/2}(f)$ that must meet $\{z:|z|<R\}$.
\end{proof}
Finally we state sufficient conditions for various sets to be spiders' webs.
\begin{lemma}\label{SWs}
Let $f$ be a {\tef}.
\begin{itemize}
\item[(a)]
If $I(f)$ contains a spider's web, then $I(f)$ is a spider's web.
\item[(b)]
If $J(f)$ contains a spider's web, then $J(f)$ is a spider's web.
\item[(c)]
If $I(f)\cap J(f)$ contains a spider's web, then $I(f)\cap J(f)$ is a spider's web.
\end{itemize}
\end{lemma}
\begin{proof}
The result in part~(a) was pointed out in \cite[Remark~1 after the proof of Theorem~1.4]{RS09a} and can also be deduced from \cite[Theorem~4.2]{RS11a}. Its proof is similar to that of part~(c) given below.

To prove part~(b) we need to show that $J(f)$ is connected. Since $J(f)$ contains a spider's web, $f$ has no multiply connected Fatou components. Therefore, all the components of $J(f)$ are unbounded (see \cite{mK98}), and hence $J(f)$ is connected.

To prove part~(c), we need to show that $I(f)\cap J(f)$ is connected. Suppose that~$E$ is a spider's web contained in $I(f)\cap J(f)$. Since $f$ can have no {\mconn} Fatou components, every component of $A(f)\cap J(f)$ is unbounded by Lemma~\ref{fast} part~(c). Hence, we can assume that $E$ contains $A(f)\cap J(f)$.

Now $E$ is contained in a single component, $I_0$ say, of $I(f)\cap J(f)$. Since $J(f)=\overline{A(f)\cap J(f)}$, by Lemma~\ref{fast} part~(b), we deduce that
\[I(f)\cap J(f)\subset \overline{A(f)\cap J(f)}\subset \overline{I_0}.\]
Thus $I(f)\cap J(f)=I_0$ and hence $I(f)\cap J(f)$ is connected, as required.
\end{proof}

\section{Proof of Theorem~\ref{main3}}
\setcounter{equation}{0}
The idea of the proof of part~(a) is as follows: we start with a curve that lies in a {\sconn} Fatou component of $f$ and we assume that the curve meets two circles of the form $C(r)$ and $C(r^L)$, for some sufficiently large values of~$r$ and~$L$. We then obtain a contradiction by showing that the forward images of this curve must either experience repeated radial stretching, and so contradict Lemma~\ref{dist}, or else some forward image of this curve must wind round~$0$ and hence meet $J(f)$. A similar argument can be used to prove part~(b)(i), so the proof of this part is done at the same time as the proof of  part~(a).

A key fact needed to obtain these contradictions is that the function $f$ has the symmetry property noted earlier
\[
f(\overline z)=\overline{f(z)},\;\;\text{for } z\in\C.
\]
This property implies that $F(f)$ is symmetric with respect to the real axis, as are the sets $A_R(f)$, where $R>0$ is such that $M(r) > r$ for $r\geq R$.

Let $m>1$ be the constant given in~\eqref{min3}, and put
\begin{equation}\label{Lsize}
L = m + 4,\;\;\eps=1/L\;\;\text{and}\;\; \mu(r)=M(r)^{1/L},\;\;r>0.
\end{equation}
Now let $R\geq \max \{1, R_0, R_1, R_2\}$, where $R_0$ is in the statement of the theorem, $R_1$ is the constant in \eqref{Had}, and $R_2$ is the constant in Lemma~\ref{AR-component}, and also let $R$ be so large that
\begin{equation}\label{Rsize}
R^{1/4} \ge \frac{384\sqrt 2\, L}{\pi}\,, \;\; J(f)\cap \{z:|z|<R\}\neq \emptyset\;  \mbox{ and } \; M(r) > r^{4L^2} \mbox{ for } r\geq R.
\end{equation}

We now start the proof of parts~(a) and (b)(i). Suppose that $U$ is either a {\sconn} Fatou component of $f$ (in part~(a)) or a component of the open set $Q_{\eps,R}(f)^c$ (in part~(b)(i)), and that for some  $r_0>R$ there exists a curve $\gamma_0$ in $U$ that meets both $C(r_0)$ and $C(r_0^L)$. To prove parts~(a) and~(b)(i), we have to obtain a contradiction in each case. As described above, we obtain a contradiction in one of two ways.

First we define the sequence $(L_n)$ by
\begin{equation}\label{Ln}
L_0=L-1 \;\;\text{and}\;\;L_{n+1} = L_n - 1/(n+1)^2,\;\;n\ge 0.
\end{equation}
Suppose that we can find curves $\gamma_n\subset f^n(\gamma_0)$ and a sequence $(r_n)$ such that, for $n=1,2,\ldots$, we have
\begin{equation}\label{eq3.1}
r_{n}\ge \mu(r_{n-1}),
\end{equation}
\begin{equation}\label{contains}
f(\gamma_{n-1})\supset \gamma_{n},
\end{equation}
\begin{equation}\label{eq3.2}
\gamma_n\subset \overline A(r_n,r_n^{L_n}),\;\text{ and }\gamma_n\text{ meets both }C(r_n) \text{ and } C(r_n^{L_n}).
\end{equation}
(These three conditions formalise what we mean by saying that the images of the curve $\gamma_0$ experience repeated radial stretching.)

We deduce, by \eqref{contains}, that there is a point $z\in \gamma_0$ such that, for $n\ge 0$,
\[
f^n(z)\in \gamma_n,\quad \text{so}\quad |f^n(z)|\ge r_n \geq \mu^n(r_0)\ge \mu^n(R),
\]
by \eqref{eq3.1} and \eqref{eq3.2}. Hence $z\in Q_{\eps,R}(f)$, which gives a contradiction in part~(b)(i).

In part~(a), we deduce that $U\subset I(f)$, so the fact that
\[
\frac{\sup \{|f^n(z)|:z\in\gamma_0\}}{\inf \{|f^n(z)|:z\in\gamma_0\}}\ge \frac{r_n^{L_n}}{r_n}\ge \frac{r_n^{L-3}}{r_n}=r_n^{m}\to\infty\;\;\text{as }n\to\infty,
\]
gives a contradiction to Lemma~\ref{dist}.

Suppose then that, for $k=0,1,\ldots, n$, curves $\gamma_k$ and positive numbers $r_k$ have been chosen such that~\eqref{eq3.1}, \eqref{contains} and~\eqref{eq3.2} are satisfied with $n$ replaced by $k$, for $k=1,\ldots, n$. To complete the proof we show that we can choose a curve $\gamma_{n+1}$ and a positive number $r_{n+1}$ so that~\eqref{eq3.1}, \eqref{contains} and~\eqref{eq3.2} hold with $n$ replaced by $n+1$.

To do this we apply Theorem~\ref{main2}, with
\[
s=r_n,\;\; l=L_n\;\;\text{and}\;\; b=1/(n+1)^2,\;\text{ where }n\ge 0,
\]
to a curve $\gamma'_n$ meeting $C(r_n)$ and $C(r_n^{L_n})$, chosen such that $\gamma'_n \subset \{z: \Im z \geq 0\}$ and
\begin{equation}\label{gamma}
\gamma'_n \subset \gamma_n \cup \gamma_n^*,
 \end{equation}
where $*$ denotes reflection in the real axis. Note that $l-b=L_{n+1}$.

We check that the hypotheses of Theorem~\ref{main2} are satisfied. By \eqref{Lsize}, \eqref{Rsize}, \eqref{Ln}, \eqref{eq3.1} and the choice of~$R$,
\[
L_n\ge L-3=m+1\ge 1+\frac{1}{(n+1)^2},
\]
\[
r_n \geq \mu^n(r_0) > \mu^n(R) > R \geq R_0,
\]
and
\[
M(r_n) \ge r_n^{4L^2}> r_n^2.
\]
Note that $\mu(r)=M(r)^{1/L} \geq r^{4L}$, for $r\ge R$, by~\eqref{Rsize}. This estimate also implies that $r_n \geq r_0^{(4L)^n}$, for $n\ge 0$, by~\eqref{eq3.1}. Hence, since $L\ge 4$, $b\le 1$ and $l-b\ge 1$, we have
\begin{eqnarray*}
s^{b/4}
&\ge& \frac{bs^{b/4}}{l-b}
= \frac{r_n^{1/(4(n+1)^2)}}{(n+1)^2 L_{n+1}}\\
&\ge& \frac{r_0^{(4L)^n/(4(n+1)^2)}}{(n+1)^2 L}
\ge \frac{r_0^{4^{n-1}}}{(n+1)^2 L}\\
&\ge& \frac{r_0^{1/4}}{L}\ge \frac{384\sqrt 2}{\pi},
\end{eqnarray*}
by \eqref{Rsize} again, as required. Thus the hypotheses of Theorem~\ref{main2} are satisfied.

If case~(1) or case~(2) of Theorem~\ref{main2} holds for $\gamma'_n$, then the same case holds for~$\gamma_n$, by the symmetry of $f$ in the real axis, and so we can choose $r_{n+1}$ and $\gamma_{n+1} \subset f(\gamma_n)$ so that they satisfy~\eqref{eq3.1}, \eqref{contains} and~\eqref{eq3.2} with~$n$ replaced by $n+1$, as required.

Thus, to complete the proof of parts~(a) and~(b)(i) it is sufficient to show that if case~(3) of Theorem~\ref{main2} holds for~$\gamma'_n$, then we obtain a contradiction.

If case~(3) holds for~$\gamma'_n$, then the image under~$f$ of some subcurve of $\gamma'_n$ is contained in
\[
\overline A(M(r_n)^{1/L_n}, M(r_n)^{L_{n+1}})\subset \{z:|z|\ge M(r_n)^{1/L_0}\}\subset \{z:|z|\ge R\},
\]
and also winds round~$0$ through an angle of at least $2\pi$. Hence, in part~(a), there is a {\mconn} component of $F(f)$ that contains $f(\gamma'_n)$,  by the symmetry of $f$ in the real axis and \eqref{Rsize}. But this is impossible by Lemma~\ref{connectivity} since
\[
f(\gamma'_n)\subset f^{n+1}(\gamma_0)\subset f^{n+1}(U),
\]
and $U$ is simply connected. This completes the proof of part~(a).

In part~(b)(i) we obtain a contradiction by using Lemma~\ref{AR-component}. Let $\tilde R=M(r_n)^{1/L_0}$. Then $\tilde R\ge R\ge R_2$ so, by Lemma~\ref{AR-component} and the symmetry of $f$ in the real axis, the set $A_{\tilde R/2}(f)\cap \{z:\Im z\ge 0\}$ has a component, $K$ say, that is unbounded and meets $C(\tilde R)$. Thus the image curve $f(\gamma'_n)$ must meet $K$. Therefore, by \eqref{contains}, there exists $z_0 \in \gamma_0$ such that $f^k(z_0)\in \gamma_k$, for $k=1,\ldots,n$, and
\begin{equation}\label{zn}
z_{n+1}=f^{n+1}(z_0) \in K.
\end{equation}
Now
\[
|f^k(z_0)|\ge r_k\ge \mu^k(R),\;\;\text{for } k=0,1,\ldots n,
\]
by \eqref{eq3.1} and \eqref{eq3.2},
\[
|f^j(z_{n+1})|\ge M^j(\tilde R/2),\;\;\text{for } j\ge 0,
\]
because $z_{n+1}\in K$, and
\[
\tilde R/2=\tfrac12 M(r_n)^{1/L_0} \ge M(r_n)^{1/L}=\mu(r_n),
\]
by \eqref{Rsize} and the fact that $L_0=L-1$. We deduce that
\[
|f^n(z_0)|\ge \mu^n(R),\;\;\text{for } n\ge 0.
\]
Thus $z_0\in Q_{\eps,R}(f)$ and this contradicts the fact that in part~(b)(i) the curve $\gamma_0$ was chosen to lie in $Q_{\eps,R}(f)^c$. So case~(3) of Theorem~\ref{main2} leads to a contradiction. This completes the proof of part~(b)(i).

Next we prove part~(b)(ii). First recall that $\eps$ and $R$ satisfy \eqref{Lsize} and \eqref{Rsize}, and that $Q_{\eps,R}(f)$ is closed and contains $A_R(f)$. For $n\in\N$, let $G_n$ denote the bounded simply connected domain obtained by taking the union of the set
\[
\{z:|z|<r'_n\}\cup \bigcup\{ U: U\;\text{is a component of } Q_{\eps,R}(f)^c, U\cap C(r'_n)\ne\emptyset\}
\]
with its bounded complementary components, where the sequence $(r'_n)$ is chosen so large that
\[
R^L\le r'_1< r'_2<\cdots,\quad r'_n\to \infty\;\text{ as } n\to\infty,
\]
and the sets $\partial G_n$ are disjoint.  This is possible by part~(b)(i). Then
\[
\partial G_n\subset Q_{\eps,R}(f),\;\;\text{for } n\in \N.
\]
Clearly $G_n \subset G_{n+1}$ for $n\in \N$ and $\bigcup_{n\in\N}G_n=\C$.

By Lemma~\ref{fast} part~(a), there exists $N\in\N$ such that $A_R(f) \cup \bigcup_{n \geq N} \partial G_n$ is a spider's web in $Q_{\eps,R}(f)$. Thus $I(f)$ is a spider's web by Lemma~\ref{SWs} part~(a).

To prove part~(c), let $G_n$ be the domains constructed above such that $\partial G_n\subset Q_{\eps,R}(f)$ and for $n\in \N$ let $G'_n$ denote the bounded simply connected domain obtained by taking the union of the set
\[
G_n\cup \bigcup\{U: U\;\text{is a Fatou component of } f, U\cap \partial G_n\ne\emptyset\}
\]
with its bounded complementary components. The domains $G'_n$ are bounded, by the result of part~(a), and we can assume by taking a subsequence that they are disjoint. Now each of the Fatou components~$U$ in the definition of $G'_n$, $n\in\N$, meets $Q_{\eps,R}(f)$ and $U$ is {\sconn} by hypothesis. Thus, by part~(a), \begin{eqnarray*}
U &\subset& \{z:|f^k(z)|\ge (\mu^k(R))^{1/L},\;\text{ for } k\ge 0\}\\
&\subset& \{z:|f^k(z)|\ge \nu^k(R),\;\text{ for } k\ge 0\},
\end{eqnarray*}
where $\nu(r)=\mu(r)^{\eps}=M(r)^{\eps^2}$, since $1/L=\eps$.

Thus, for such a Fatou component $U$, we have $U\subset Q_{\eps^2,R}(f)$. Since $Q_{\eps,R}(f) \subset Q_{\eps^2,R}(f)$, and both $Q_{\eps^2,R}(f)$ and $J(f)$ are closed, we deduce that
\[
\partial G'_n\subset Q_{\eps^2,R}(f)\cap J(f),\;\;\text{for } n\in \N.
\]
Clearly $G'_{n}\subset G'_{n+1}$ for $n\in\N$ and $\bigcup_{n\in\N}G'_n=\C$.

Hence, by Lemma~\ref{fast} part~(c), $Q_{\eps^2,R}(f)\cap J(f)$ contains a spider's web. Thus both $I(f)\cap J(f)$ and $J(f)$ contain spiders' webs, so the result of part~(c) follows from Lemma~\ref{SWs} parts~(b) and~(c). This completes the proof of Theorem~\ref{main3}.


\begin{thebibliography}{99}

\bibitem{iB58} I.N. Baker, Zusammensetzungen ganzer Funktionen,
{\it Math. Z.}, 69 (1958), 121--163.

\bibitem{iB81} I.N. Baker, The iteration of polynomials and transcendental entire functions, {\it J. Austral. Math. Soc. (Series A)}, 30 (1981), 483--495.

\bibitem{iB84} I.N. Baker, Wandering domains in the iteration of entire functions,
{\it Proc. London Math. Soc.} (3), 49 (1984), 563--576.

\bibitem{Ba} P.D. Barry, On a theorem of Besicovitch,
{\it Quart. J. Math. Oxford Ser.} (2), 14 (1963), 293--302.

\bibitem{wB93} W. Bergweiler, Iteration of meromorphic functions, {\it Bull. Amer. Math. Soc.}, 29 (1993), 151--188.

\bibitem{E} A.E. Eremenko, On the iteration of entire functions, {\it Dynamical systems and ergodic theory,} Banach Center Publications 23, Polish Scientific Publishers, Warsaw, 1989, 339--345.

%

\bibitem{HK} W.K. Hayman and P.B. Kennedy, {\it Subharmonic functions, Volume 1}, Academic Press, 1976.

\bibitem{wH89} W.K. Hayman, {\it Subharmonic functions, Volume 2}, Academic Press, 1989.

\bibitem{H} A. Hinkkanen, Entire functions with bounded Fatou components, {\it Transcendental dynamics and complex analysis}, 187--216, Cambridge University Press, 2008.

\bibitem{HM} A. Hinkkanen and J. Miles, Growth conditions for entire functions with only bounded Fatou components, {\it Journal d'Analyse Math.}, 108 (2009), 87--118.

\bibitem{mK98} M. Kisaka, On the connectivity of Julia sets of {\tef}s, {\it Ergodic Theory Dynam. Systems}, 18 (1998), 225--250.

\bibitem{New} M.H.A. Newman, {\it Elements of the topology of plane sets of points},
Cambridge University Press, 1961.

\bibitem{lR07} L.\ Rempe, On a question of Eremenko concerning escaping components of entire functions, {\it Bull.\ London Math.\ Soc.}, 39 (2007), 661--666.


\bibitem{RS05} P.J. Rippon and G.M. Stallard, On questions of Fatou and Eremenko,
{\it Proc. Amer. Math. Soc.}, 133 (2005), 1119--1126.

\bibitem{RS08} P.J. Rippon and G.M. Stallard, Functions of small growth with no
unbounded Fatou components, {\it Journal d'Analyse Math.}, 108 (2009), 61--86.

\bibitem{RS11a} P.J. Rippon and G.M. Stallard, Boundaries of escaping Fatou components, {\it Proc. Amer. Math. Soc.}, 139 (2011), 2807--2820.

\bibitem{RS09a} P.J. Rippon and G.M. Stallard, Fast escaping points of entire functions, To appear in {\it Proc. London Math. Soc.}, arXiv:1009.5081.

\bibitem{RS11} P.J. Rippon and G.M. Stallard, Dynamics of meromorphic functions with direct tracts: fast escaping spiders' webs. In preparation.

\bibitem{RS11b} P.J. Rippon and G.M. Stallard, Regularity and fast escaping points of entire functions. In preparation.


\bibitem{RRRS} G.\ Rottenfu{\ss}er, J. R\"uckert, L. Rempe and D.\ Schleicher,
Dynamic rays of bounded-type entire functions, {\it Ann. of Math.}, 173 (2011), 77--125.

\bibitem{nS93} N. Steinmetz, {\em Rational iteration}, de Gruyter, 1993.


\bibitem{Z} Jian-Hua Zheng, Unbounded domains of normality of entire functions of small growth, {\it Math. Proc. Camb. Phil. Soc.}, 128 (2000), 355--361.



\end{thebibliography}
\end{document}